\documentclass[11pt]{article}

\usepackage{amsmath,amsfonts,amsthm,amssymb,blkarray, booktabs}
\usepackage{mathrsfs, graphicx,color,supertabular}
\usepackage{ulem}
\usepackage{cite}
\usepackage[T1]{fontenc}
\newtheorem{thm}{Theorem}[section]
\newtheorem{lem}[thm]{Lemma}

\newtheorem{exa}{Example}

\allowdisplaybreaks

\title{\bf \large All limit points of the largest roots of matching polynomials are determined}

\author{
{\small \ \ Zhaoxi Li$^a$,\ \ Shi-Mei Ma$^a$,\ \ Jianfeng Wang$^{a,}$\footnote{Corresponding author.
\newline{\it \hspace*{5mm}Email addresses:} lzhaoxi@aliyun.com (Z.X. Li), shimeimapapers@163.com (S.-M. Ma), jfwang@aliyun.com (J.F. Wang).}\;, \ \ Weifan Wang$^{a,b}$}\\[2mm]
\footnotesize $^a$School of Mathematics and Statistics, Shandong University of Technology, Zibo 255049, China\\
\footnotesize $^b$Department of Mathematics, Zhejiang Normal University, Jinhua, 321004, China}

\date{ }
\begin{document}
\maketitle
\begin{abstract}
The largest matching root $\mu(G)$ of a graph $G$ is that of its matching polynomial. In this paper, all limit points of the largest matching roots of graphs are determined. More precisely, we identify the limit points of the largest matching roots of graphs less than $\tau^{\frac{1}{2}}+\tau^{-\frac{1}{2}}$. For any $\gamma \geq \tau^{\frac{1}{2}}+\tau^{-\frac{1}{2}}$ with $\tau=\frac{\sqrt{5}+1}{2}$, there exists a graph sequence $\{G_i\, |\, i\in \mathbb{N}\}$ such that $\lim\limits_{i \rightarrow \infty}\mu(G_i)=\gamma$.\\

\end{abstract}

\baselineskip=0.22in

\section{Introduction}
In this paper, all graphs considered are simple and undirected. Let $G=(V(G),E(G))$ be a graph with order $n=|V(G)|$ and size $m = |E(G)|$. An edge subset $M \subseteq E(G)$ is called a {\it matching} of $G$ if no two distinct edges in it share a common vertex, and let $\mathcal{M}(G)$ be the set of matchings of $G$. As a celebrated graph polynomial, the {\it matching polynomial} of $G$ is defined as
$$\mathscr{M}(G,x)=\sum_{M\in \mathcal{M}(G)}(-1)^{|M|}x^{|V(G)\backslash V(M)|}.$$
In early period,  the matching polynomial was investigated as a mathematical model in statistical physics \cite{gru-kun,hei-lie,kun} and theoretical chemistry \cite{Aih,gut-mil-tri1, gut-mil-tri} as well as appeared in the first mathematical literature \cite{far}. Since then, the study of matching polynomials has garnered considerable attention. For its combinatorial properties, one of amazing results is that infinitely many bipartite Ramanujan graphs were constructed via matching polynomial\cite{mar-spi-sri}. For more results, see the monographs \cite{god-book,shi-GP} and some old and new papers \cite[e.g.]{bal,bur-erey,chen-yuan,chen-yuan1,god-gut,Gutman1,ku-chen,yan-yeh}.

The {\it matching roots} of a graph $G$ are those of $\mathscr{M}(G,x)$. It has been proved that all the matching roots of $G$ are real and lie in the interval $(-2\sqrt{\Delta-1},2\sqrt{\Delta-1})$ with $\Delta = \Delta(G)$ being the maximum degree of $G$ \cite{hei-lie}. Since the roots of graph polynomials are closely related to the structures of graphs, of particular interest here is the largest matching root of $G$, denoted by $\mu(G)$. Let $A(G)=(a_{ij})$ be the {\it adjacency matrix} of $G$ with $a_{ij} =1$ if $v_iv_j \in E(G)$ and othewise $a_{ij}=0$. The {\it characteristic polynomial} of $G$ is denoted by  $\phi(G,\lambda) = \det(\lambda I - A(G))$. As usual, the {\it spectral radius} $\rho_A(G)$ of $G$ is the largest absolute value of all roots of $\phi(G,\lambda)$. It is well-known that if $G$ is connected, then $\mu(G)$ is simple and non-negative \cite{god-book,hei-lie}, and $\mathscr{M}(G,x) = \phi(G,\lambda)$ if and only if $G$ is a tree \cite[eg.]{god-gut-JGT}. Thereby, for a  tree,
\begin{equation}\label{mu=rho}
	\mu(G) = \rho_A(G).
\end{equation}

A real number $\gamma$ is called a {\it limit point} of the largest matching roots of graphs if there exists a sequence of graphs $\{G_i\, |\, i\in \mathbb{N}\}$ such that  $$\mu(G_i) \neq \mu(G_j) \quad \text{whenever  $i \neq j$}, \quad \text{and}  \quad \lim_{i \rightarrow \infty}\mu(G_i)=\gamma.$$
Recall, the research about this aspect originated from A.J. Hoffman \cite{Ho A J}, who detected the limit points of spectral radii of nonnegative symmetric integral matrices. Using Hoffman's results, Jiang and Polyanskii \cite{jiang-poly} and Jiang et al \cite{jiang-tid-yao-zhang} had attacked a longstanding problem on equiangular lines. We refer the readers to a quite recent survey \cite{3w-2b} in this direction.

In the paper, we determine all limit points of the largest matching roots of graphs in the following theorem, which enriches the theory of matching polynomials.

\begin{thm}\label{main}
Set $\tau=\frac{\sqrt{5}+1}{2}$. Let $\alpha_{n}=\beta_{n}^{\frac{1}{2}}+\beta_{n}^{-\frac{1}{2}}$ with $\beta_{n}$ being the largest positive root of $h_{n}(x)=x^{n+1}-(1+x+\cdots+x^{n-1})$. Then
$$\{\alpha_{1},\alpha_{2},\ldots\}\cup [\tau^{\frac{1}{2}}+\tau^{-\frac{1}{2}},+\infty)$$
constitutes all limit points of the largest matching roots of graphs over $\mathbb{R}$, where $2=\alpha_{1}<\alpha_{2}<\cdots$ and $\lim\limits_{n \to \infty}\alpha_{n}=\tau^{\frac{1}{2}}+\tau^{-\frac{1}{2}} = \sqrt{2+\sqrt{5}}$.
\end{thm}

In Section 2, we first prove that any real number $\gamma \geq \tau^{\frac{1}{2}}+\tau^{-\frac{1}{2}}$ is a limit point of the largest matching roots of graphs. To determine all limit points of the largest matching roots of graphs smaller than $\tau^{\frac{1}{2}}+\tau^{-\frac{1}{2}}$, we show that each  graph in $\{G_i | i \in \mathbb{N}\}$ must be a tree with maximum degree $\Delta(G_i) =3$. In Section 3, we further analyse the concrete structure of such trees in the sequence. In Section 4, we complete the  algebraic computation to establish the formula given in Theorem \ref{main}.

\section{Reduction to trees}

To prove the main result, we will see that the required graph sequence consists of trees. In particular, the second part of Theorem 1.1 relies on a celebrated result of Shearer [31].

\begin{thm}\label{ii}
 Any real number $\gamma \geq \tau^{\frac{1}{2}}+\tau^{-\frac{1}{2}}$ is a limit point of the largest matching roots of graphs.
\end{thm}

\begin{proof}
For any real number $\gamma \geq \tau^{\frac{1}{2}}+\tau^{-\frac{1}{2}}$, Shearer \cite{she} constructed a tree sequence to verify this theorem for adjacency spectral radii $\rho_A(G)$. Thereby, it also fits for the largest matching roots  by \eqref{mu=rho}. In fact, such a sequence of trees is formed from the so-called caterpillars in which the removal of all pendant vertices results in a path.	
\end{proof}

To determine all limit points of the largest matching roots of graphs smaller than $\tau^{\frac{1}{2}}+\tau^{-\frac{1}{2}}$, we introduce some basic lemmas and useful results.

For a connected graph $G$ and $v\in V(G)$, Godsil \cite{C.D.Godsil2} defined the \textit{path-tree} $T(G,v)$ of $G$ to be a tree which has as vertices the paths in $G$ that start at $v$ and where two such paths are adjacent if one is a maximal proper sub-path of the other. We now present an example of a path-tree with respect to a graph. Let $P_n$ and $C_n$ be respectively the {\it path} and {\it cycle} with order $n$.
\begin{exa}\normalfont
 Consider the path-tree $T(C_{4},1)$ of $C_{4}$. Note that $$V\left(T(C_{4},1)\right)=\{1,12,123,1234,14,143,1432\}.$$ Then $T(C_{4},1)=P_{7}$, which is shown in Fig. 1.
	
 \setlength{\unitlength}{1mm} \linethickness{1pt}	
	\begin{center}
		\begin{picture}(180,50)(2,-10)
			\multiput(30,20)(20,0){2}{\circle*{1.5}}
			\multiput(30,40)(20,0){2}{\circle*{1.5}}
			
			\put(30,20){\line(3,0){20}}\put(30,20){\line(0,3){20}}
			\put(30,40){\line(3,0){20}}\put(50,20){\line(0,3){20}}	
			\put(26,40){$1$} \put(52,40){$2$}
			\put(26,18){$4$} \put(52,18){$3$}
			\put(39,12){\large$C_{4}$}
			\put(68,33){\large$T(C_{4},1)$}
			\put(65,30){\vector(1,0){20}}
			
			\put(115,40){\circle*{1.5}}
			\multiput(110,33.4)(10,0){2}{\circle*{1.5}}
			\multiput(105,26.6)(20,0){2}{\circle*{1.5}}
			\multiput(100,20)(30,0){2}{\circle*{1.5}}
			
			\put(100,20){\line(3,4){15}}
			\put(130,20){\line(-3,4){15}}
			\put(89,22){ $1234$ }
			\put(96,28){ $123$ }
			\put(102,33){ $12$ }
			\put(115,41){  $1$ }
			\put(129,22){ $1432$ }
			\put(125,28){ $143$ }
			\put(121,33){ $14$ }
			\put(115,12){\large$P_{7}$}
			
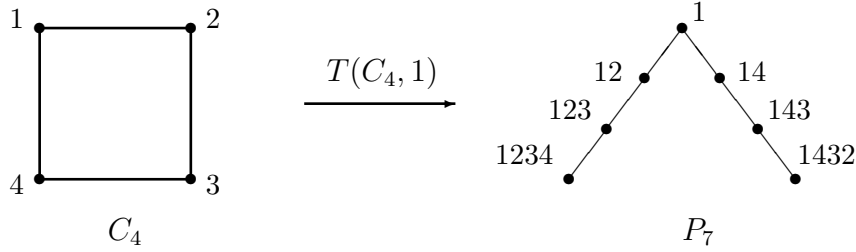
			\put(50,5){{\bf Fig. 1:} $C_{4}$ and its path-tree $T(C_{4},1)$.}		
		\end{picture}
	\end{center}
\end{exa}
\vspace{-16mm}

Let $H\subset G$ denote that $H$ is a proper subgraph of a graph $G$.

\begin{lem}\label{largest root}
Let $G$ be a connected graph and $H\subset G$. Then
\begin{itemize}
\item[{\rm (i)}]{\rm\cite{book}}
$\;\rho_{A}(H)<\rho_{A}(G)$.
\item[{\rm (ii)}]{\rm\cite{C.D.Godsil2,hei-lie}}
$\mu(H)<\mu(G)<2\sqrt{\Delta(G)-1} \; \mbox{and} \; \rho_{A}(T(G,v))=\mu(G).$
\end{itemize}	
\end{lem}

Hoffman and Smith \cite{Ho and Sm} defined an {\it internal path} of a graph $G$ as a walk $v_0,v_1,\cdots,v_k(k \geq 1)$ where the vertices $v_1,\cdots,v_k$ are distinct ($v_0, v_k$ need not be distinct), $d_{G}(v_0) > 2$, $d_{G}(v_k) > 2$ and $d_{G}(v_i) = 2$ whenever $0 < i < k$. In terms of the definition, two types of internal paths (see Fig. 2) are shown as follows:\smallskip

Type A. $v_0= v_k$, $d_{G}(v_0) \geq 3$ and $d_{G}(v_i) = 2$ for $0 < i < k$.\smallskip

Type B.  $v_0 \neq v_k$, $d_{G}(v_0)\geq 3$, $d_{G}(v_k) \geq 3$ and $d_{G}(v_i) = 2$ for $0 < i < k$.\smallskip

\begin{center}
\setlength{\unitlength}{1mm} \linethickness{1pt}
\begin{picture}(155,34)
\qbezier[5](15.5,29.8)(18,32)(21,32)
\qbezier[7](14.5,22.5)(15,18)(21,17.7) \put(28,25){\circle*{1.5}}
\multiput(34,29)(0,-8){2}{\circle*{1.5}}
\multiput(33.7,23)(0,1.8){3}{\line(0,1){0.5}}
\put(28,25){\line(3,2){6}} \put(28,25){\line(3,-2){6}}
\put(24,24.5){$v_{_0}$} \put(15.5,29.8){\circle*{1.5}}
\put(17,28){$v_t$} \put(14.5,22.5){\circle*{1.5}}
\put(16,22){$v_{t+{\!}1}$} \linethickness{0.6pt}
\qbezier(15.5,29.8)(13,26)(14.5,22.5) \put(21,32){\circle*{1.5}}
\put(21,17.7){\circle*{1.5}} \qbezier(21,32)(28,31)(28,25)
\qbezier(21,17.7)(28,20)(28,25) \put(17,12){\mbox{Type A}}

\multiput(60,20)(0,10.1){2}{\circle*{1.5}}
\multiput(68,25)(9,0){2}{\circle*{1.5}}
\multiput(85,25)(10,0){3}{\circle*{1.5}}
\multiput(115,25)(9,0){2}{\circle*{1.5}}
\multiput(132,20)(0,10){2}{\circle*{1.5}} \linethickness{1pt}
\multiput(62,23)(0,1.8){3}{\line(0,1){0.5}}
\multiput(130,23)(0,1.8){3}{\line(0,1){0.5}}
\put(60,20){\line(3,2){8}} \put(68,25){\line(-3,2){8}}
\linethickness{0.6pt} \put(68,25.3){\line(1,0){9}}
\multiput(77,25.3)(1.5,0){28}{\line(1,0){0.5}}
\put(84,25.3){\line(1,0){20}} \put(115,25.3){\line(1,0){9}}
\put(124,25){\line(3,-2){8}} \put(124,25){\line(3,2){8}}
\put(67.5,27.5){$v_{_0}$} \put(75.5,27.5){$v_{_1}$}
\put(83.5,27.5){$v_{_{t-1}}$} \put(93.5,27.5){$v_{_t}$}
\put(103,27.5){$v_{_{t+1}}$} \put(112,27.5){$v_{_{k-1}}$}
\put(122,27.5){$v_{_k}$} \put(90,12){Type B}
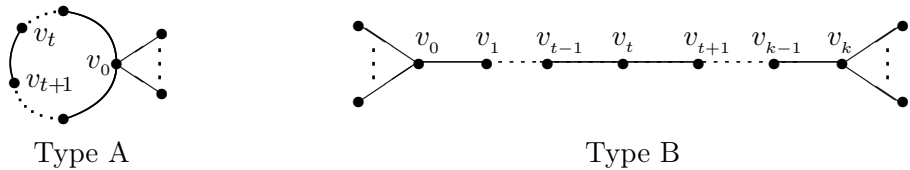
\put(50,3){{\bf Fig. 2:} Two types of internal path.}
\end{picture}
\end{center}
 \vspace{-4mm}

In Spectral Graph Theory, the properties of internal path have been studied in detail (see Lemma \ref{adja subd}).  Contrary to previous views, we here provide a novel conclusion for the largest matching roots of graphs (see Lemma \ref{3.1}). Let $F_{s,r}$ $(s>1,r\geq1)$ be the tree obtained from $P_{s+r+1}$: $v_{0}v_{1}\cdots v_{s}u_{1}\cdots u_{r}$ and two isolated vertices $w_{1},w_{2}$ by adding edges $v_{1}w_{1}$ and $v_{s}w_{2}$.
\begin{lem}{\rm\cite{Ho and Sm}}\label{adja subd}
	Let $uv$ be an edge of the connected graph $G$ and let $G_{uv}$ be obtained from $G$ by subdividing $uv$.
	\begin{itemize}
		\item[{\rm (i)}]
		If $uv$ is not in an internal path and $G\neq C_{n}$, then $\rho_{A}(G_{uv})>\rho_{A}(G).$
		\item[{\rm (ii)}]
		If $uv$ is in an internal path and $G\neq F_{s,1}$, then $\rho_{A}(G_{uv})<\rho_{A}(G).$
	\end{itemize}
\end{lem}

\begin{lem}\label{3.1}
	Let $uv$ be an edge of the connected graph $G$ and let $G_{uv}$ be obtained from $G$ by subdividing $uv$.
	\begin{itemize}
		\item[{\rm (i)}]
		If $uv$ is not in an internal path or $uv$ is in an internal path of Type A, then $$\mu(G_{uv})>\mu(G).$$
		\item[{\rm (ii)}]
		If $uv$ is a cut edge in an internal path of Type B and $G\neq F_{s,1}$, then $$\mu(G_{uv})<\mu(G).$$
	\end{itemize}
\end{lem}
\begin{proof}
	(i) For the former, if $G\neq C_{n}$, then $G\subset G_{uv}$, and so $\mu(G_{uv})>\mu(G)$ by Lemma \ref{largest root}(ii). If $G=C_{n}$, then the path-tree $T(G,u)=P_{2n-1}$ is a proper subgraph of path-tree $T(G_{uv},u)=P_{2n+1}$. By Lemma \ref{largest root}(i) and (ii) we get $\rho_A(T(G_{uv},u)) > \rho_A(T(G,u))$, and thus $\mu(G_{uv})>\mu(G)$.
	
	For the latter, as labeled in Type A of Fig. 2, we get $T(G,v_{0}) \subset T(G_{uv},v_{0})$. Hence,  $\mu(G_{uv})>\mu(G)$ follows by Lemma \ref{largest root}(i) and (ii).
	
	(ii) Since $uv$ is a cut edge in an internal path of Type B, all edges in this internal path are cut edges. Hence, the path-tree $T(G_{uv},v_{0})$ is obtained from $T(G,v_{0})$ by subdividing an edge of the internal path, as labeled in Type B of Fig. 2. Then $\mu(G_{uv})<\mu(G)$ by Lemma \ref{largest root}(ii).
\end{proof}

We next introduce a formula for computing matching polynomials. For a subset $W\subseteq V(G)$, we use $G[W]$ to denote the subgraph of $G$ induced by $W$, and use $G-W$ to denote the induced subgraph $G[V(G)\backslash V(W)]$.

\begin{lem}{\rm \cite{god-book}}
	Let $G$ be a graph and $v\in V(G)$. Then
	\begin{equation}\label{0}
		\mathscr{M}(G,x)=x\mathscr{M}(G-v,x)-\sum_{u\in N_{G}(v)}\mathscr{M}(G-v-u,x),
	\end{equation}
	where $N_{G}(v)$ is the set of all vertices in $V(G)$ adjacent to $v$.
\end{lem}

\begin{lem}\label{2.9}
	Let $a=\frac{x+\sqrt{x^{2}-4}}{2}$ and  $b=\frac{x-\sqrt{x^{2}-4}}{2}$. Then
	$$\mathscr{M}(P_{n},x)=\frac{1}{\sqrt{x^{2}-4}}(a^{n+1}-b^{n+1}).$$
\end{lem}
\begin{proof}
	We employ induction on $n$. It is easy to verify the result for $n=1,2$. Suppose that the result holds for $1\leq k\leq n$ with $n\geq3$. When $k=n+1$, it follows from Equality \eqref{0}  that
	\begin{align*}
		\mathscr{M}(P_{n+1},x)&=x\mathscr{M}(P_{n},x)-\mathscr{M}(P_{n-1},x)\\
		&=\frac{x}{\sqrt{x^{2}-4}}\left(a^{n+1}-b^{n+1}\right)-\frac{1}{\sqrt{x^{2}-4}}\left(a^{n}-b^{n}\right)\\
		&=\frac{1}{\sqrt{x^{2}-4}}\left(a^{n}(ax-1)-b^{n}(bx-1)\right)\\
		&=\frac{1}{\sqrt{x^{2}-4}}\left(a^{n+2}-b^{n+2}\right).	
	\end{align*}
	This finishes the proof.
\end{proof}

Together with Equality \eqref{0} and Lemma \ref{2.9}, we determine the limit point of the largest matching roots of  a graph sequence obtained by adding $k$ new path $P_{n}$ on a vertex of $G$. We will label $\mathscr{M}(G,x)$ as $\mathscr{M}(G)$ if it is clear from text.
\begin{lem}\label{k largest root equation}
	Let $G$ be a connected graph with $v\in V(G)$. $L_{v}(G;kP_{n})$ is obtained from $G$ and $k$ new paths $P_{n}$ by adding $k$ new edges $v_{i}v$ with $i=1,\cdots, k$, respectively, where $v_{i}$ is an end vertex of $i$-th path $P_{n}$. Then
	\begin{itemize}
		\item[{\rm (i)}]
		$\lim\limits_{n \to \infty}\mu(L_{v}(G;kP_{n}))=r$ exists and $r\geq 2$.
		\item[{\rm (ii)}] If $r>2$, then $r$ is the largest positive root of the equation:
		\begin{equation*}
			\frac{x+\sqrt{x^{2}-4}}{2k}\mathscr{M}(G,x)-\mathscr{M}(G-v,x)=0.
		\end{equation*}
	\end{itemize}
\end{lem}
\begin{proof}
Since $L_{v}(G;kP_{n})\subset L_{v}(G;kP_{n+1})$, then
$$\mu(L_{v}(G;kP_{n}))< \mu(L_{v}(G;kP_{n+1}))<\max \{2,2\sqrt{\Delta(G)+k-1}\},$$
and thus $\lim\limits_{n \to \infty}\mu(L_{v}(G;kP_{n}))=r$ exists by Lemma \ref{largest root}(ii). By $P_{n}\subset L_{v}(G;kP_{n})$ and Lemma \ref{Pn Cn} we have $r\geq 2$.	
	
For simplicity, set $G'=L_{v}(G;kP_{n})$. In the light of \eqref{0},
\begin{align*}
	\mathscr{M}(G')=&x\mathscr{M}(G'-v)-\sum_{u\in N_{G'}(v)}\mathscr{M}(G'-v-u)\\[2mm]
	=&x\mathscr{M}(G-v)\mathscr{M}^{k}(P_{n})-\sum_{u\in N_{G}(v)}\mathscr{M}(G-v-u)\mathscr{M}^{k}(P_{n})-k\mathscr{M}(G-v)\mathscr{M}^{k-1}(P_{n})\mathscr{M}(P_{n-1})\\[2mm]
	=&\mathscr{M}(G)\mathscr{M}^{k}(P_{n})-k\mathscr{M}(G-v)\mathscr{M}^{k-1}(P_{n})\mathscr{M}(P_{n-1})\\[2mm]
	=&k\mathscr{M}^{k-1}(P_{n})\mathscr{M}(P_{n-1})\left(\frac{\mathscr{M}(P_{n})}{k\mathscr{M}(P_{n-1})}\mathscr{M}(G)-\mathscr{M}(G-v)\right).
\end{align*}	
For $x>2$, it is easy to verify that $\frac{b}{a}<1$. By Lemma \ref{2.9} we have
\begin{align*}
	\lim_{n \to \infty}\left(\frac{\mathscr{M}(P_{n})}{k\mathscr{M}(P_{n-1})}\mathscr{M}(G)-\mathscr{M}(G-v)\right)
	&=\dfrac{a}{k}\mathscr{M}(G)-\mathscr{M}(G-v)\\[2mm]
	&=\frac{x+\sqrt{x^{2}-4}}{2k}\mathscr{M}(G)-\mathscr{M}(G-v).
\end{align*}
Since $r>2$, we conclude that $r$ is the largest positive root of the equation:
\begin{equation*}
	\frac{x+\sqrt{x^{2}-4}}{2k}\mathscr{M}(G,x)-\mathscr{M}(G-v,x)=0.
\end{equation*}
This finishes the proof.		
\end{proof}

In the following, we focus on the limit points of the largest matching roots of graphs smaller than $\tau^{\frac{1}{2}}+\tau^{-\frac{1}{2}}$. From the fact that $\mu(G) = \max\{\mu(G_1),\mu(G_2)\}$ for $G = G_1 \cup G_2$, we can suppose that each graph in $\{G_i | i \in \mathbb{N}\}$ is connected. Set $\mu(G_{i})\neq\mu(G_{j})$ for $i\neq j$ and $\lim\limits_{i \to \infty}\mu(G_i)=r < \tau^{\frac{1}{2}}+\tau^{-\frac{1}{2}}$.

We first claim that $\Delta(G_{i})$ is bounded and $\lim\limits_{i \to \infty}d(G_{i})= \infty$, where $d(G_{i})$ is the diameter of $G_{i}$. Notice that for the star $K_{1,i}$,  we have $\lim\limits_{i \to \infty}(\mu(K_{1,i})) = \lim\limits_{i \to \infty}\sqrt{i} = \infty$ by \eqref{mu=rho}. Together with the inequality $|V(G_{i})|\leq \Delta(G_{i})^{d(G_{i})}+1$ (see \cite{Ho A J}), we obtain the desired claim.
	
For our goal, we proceed to discuss by the values of $\Delta(G_{i})$. Actually, we can claim that $\Delta(G_{i})\leq 3$. For this purpose, the following results are needed.

\begin{center}
	\setlength{\unitlength}{1mm} \linethickness{1pt}
	\begin{picture}(80,45)
		
		\multiput(8,27.5)(12,0){6}{\circle*{1.5}}
		\linethickness{0.6pt}
		\multiput(20,15.5)(0,24){2}{\circle*{1.5}}
		\put(20,15.5){\line(0,3){24}}
		\put(8,27.5){\line(3,0){38}}
		\put(54,27.5){\line(3,0){14}}

		\multiput(48,27.5)(1.5,0){12}{\line(1,0){0.5}}

		\put(7,23){$v$} \put(16,23){$v_{1}$}\put(16,15){$u$}\put(16,39){$w$}
		\put(28,23){$v_{2}$}\put(40,23){$v_{3}$}\put(52,23){$v_{n-1}$}\put(64,23){$v_{n}$}
		
		\put(5,4){{\bf Fig. 3:} The graph $S_{n}(n\geq1)$.}
	\end{picture}
\end{center}

\vspace{-6mm}

\begin{lem}\label{Sn}
	Let $S_{n}$ be obtained from $P_{n}:v_{1}v_{2}\cdots v_{n-1}v_{n}$ and three isolated vertices $u, v$ and $w$ by adding three new edges $uv_{1}, vv_{1}$ and $wv_{1}$, which is shown in Fig. 3. Then
	$$\lim\limits_{i \to \infty}\mu(S_{i})=\dfrac{3}{\sqrt{2}}>\tau^{\frac{1}{2}}+\tau^{-\frac{1}{2}}.$$
\end{lem}

\begin{proof}
	Due to $S_{i} \subset S_{i+1}$ and $\mu(S_{2})=2$, then  $2<\mu(S_{i})<\mu(S_{i+1})<2\sqrt{3}$ for $i>2$. Thus,  $\lim\limits_{i \to \infty}\mu(S_{i})=r > 2$ exists. In terms of Lemma \ref{k largest root equation}(ii), $r$ is the largest root of the equation:
	$$\frac{x+\sqrt{x^{2}-4}}{2}\mathscr{M}(S_{1},x)-\mathscr{M}(S_{1}-v_{1},x)=0.$$
	Since $\mathscr{M}(S_{1},x)=x^{4}-3x^{2}$ and $\mathscr{M}(S_{1}-v_{1},x)=x^{3}$, the previous equation is equivalent to
	$$\frac{x+\sqrt{x^{2}-4}}{2}(x^{4}-3x^{2})-x^{3}=0, \; {\mbox{i.e.},}\;2x^{2}-9=0,$$
	which implies $\lim\limits_{i \to \infty}\mu(S_{i})=\dfrac{3}{\sqrt{2}}$.
\end{proof}
\begin{lem}\label{Pn Cn}
	Let $\{G_{i} | i \in \mathbb{N}\}$ be a graph sequence. For $G_{i} \in \{P_{i}, C_{i}\}$,  then $\lim\limits_{i \to \infty}\mu(P_{i})=\lim\limits_{i \to \infty}\mu(C_{i})=2$.
\end{lem}
\begin{proof}
	Due to $P_{i}\subset P_{i+1}$ and $\Delta(P_{i})\leq2$, then by Lemma \ref{largest root}(ii) we get $\mu(P_{i})<\mu(P_{i+1})<2$. Thereby, $\mu(P_{i})=\rho_A(P_{i})=2\cos\frac{\pi}{i+1}$ (see \cite{Brouwer}) and $\mu(C_{i})=\rho_{A}(P_{2i-1})=2\cos \frac{\pi}{2i}$ leading to the lemma.	
\end{proof}

We are now in the position to show $\Delta(G_{i})\leq 3$. Otherwise if $\Delta(G_{i})\geq4$ for infinitely many $i's$, then $S_{n_{i}} \subseteq G_{i}$ (see Fig. 3). By Lemma \ref{largest root}(ii) and Lemma \ref{Sn}, we get
$$\lim\limits_{i \to \infty}\mu(G_{i})\geq \lim\limits_{i \to \infty}\mu(S_{n_{i}})>\tau^{\frac{1}{2}}+\tau^{-\frac{1}{2}},$$ a contradiction.

If $\Delta(G_{i})=2$ for infinitely many $i's$, then $G_{i}$ is a path or a cycle. It follows from Lemma \ref{Pn Cn} that
$\lim\limits_{i \to \infty}\mu(G_i)=\lim\limits_{i \to \infty}\mu(P_{i})=\lim\limits_{i \to \infty}\mu(C_{i})=2.$

We now consider $\Delta(G_{i})=3$. In this event, we claim that the graph sequence $\{G_i | i \in \mathbb{N}\}$ must be a tree sequence. To verify it, we need the following lemmas.

\begin{center}
	\setlength{\unitlength}{1mm} \linethickness{1pt}
	\begin{picture}(100,34)
		\qbezier[5](15.5,29.8)(18,32)(20,32)
		\qbezier[7](14.5,22.5)(15,18)(21,17.7)
		\put(28,25){\circle*{1.5}}
		\put(19,23.5){$C_{g}$} \put(15.5,29.8){\circle*{1.5}}
		\put(14.5,22.5){\circle*{1.5}}
		\linethickness{0.6pt}
		\qbezier(15.5,29.8)(13,26)(14.5,22.5) \put(21,32){\circle*{1.5}}
		\put(21,17.7){\circle*{1.5}}
		\multiput(38,25)(10,0){4}{\circle*{1.5}}
		\qbezier(21,32)(28,31)(28,25)
		\qbezier(21,17.7)(28,20)(28,25)
		
		\multiput(48,25)(1.5,0){7}{\line(1,0){0.5}}
		\put(28,25){\line(3,0){22}}\put(56,25){\line(3,0){12}}
		
		\put(29,21){$v$}\put(39,21){$v_{_{1}}$}\put(49,21){$v_{_{2}}$}\put(59,21){$v_{_{t-1}}$}\put(69,21){$v_{_{t}}$}
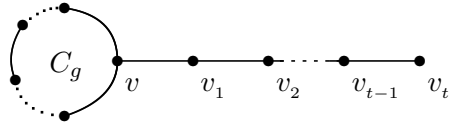
		\put(12,5){{\bf Fig. 4:} The graph $H_{g,t}(g\geq3, t\geq1)$.}
	\end{picture}
\end{center}

\vspace{-6mm}

\begin{lem}\label{3.2}
	Let $H_{g,t}$ be the lollipop graph with the labelling of Fig. 4. Then
	$$\lim_{g \to \infty}\mu(H_{g,1})=\tau^{\frac{1}{2}}+\tau^{-\frac{1}{2}}.$$
\end{lem}

\begin{proof}
It follows from Lemma \ref{3.1}(i) that $\{\mu(H_{g,1}) | g \in \mathbb{N}\; \text{and}\; g \geq 3\}$ is a strictly increasing sequence. Due to $\mathscr{M}(H_{4,1},x)=2$ and Lemma \ref{largest root}(ii), then  $\lim\limits_{g \to \infty}\mu(H_{g,1})=r > 2$ exists.

For simplicity, let $H'=H_{g,1}$ and $v$ be the vertex of $C_{g}$ with $d_{H'}(v)=3$. In the light of \eqref{0},
\begin{align*}
	\mathscr{M}(H')&=x\mathscr{M}(H'-v)-\sum_{u\in N_{H'}(v)}\mathscr{M}(H'-v-u)\\
	&=x\mathscr{M}(P_{g-1})\mathscr{M}(P_{1})-2\mathscr{M}(P_{g-2})\mathscr{M}(P_{1})-\mathscr{M}(P_{g-1})\\[2mm]
	&=(x^2-1)\mathscr{M}(P_{g-1})-2x\mathscr{M}(P_{g-2})\\[2mm]
	&=\mathscr{M}(P_{g-2})\left[\frac{\mathscr{M}(P_{g-1})}{\mathscr{M}(P_{g-2})}(x^2-1)-2x\right].
\end{align*}	
From Lemma \ref{2.9}, we have, for $x>2$, that
\begin{align*}
	\lim_{g \to \infty}\left[\frac{\mathscr{M}(P_{g-1})}{\mathscr{M}(P_{g-2})}(x^2-1)-2x\right]&=a(x^2-1)-2x\\&
	=\frac{x+\sqrt{x^{2}-4}}{2}(x^2-1)-2x,
\end{align*}
which implies that $\lim\limits_{g \to \infty}\mu(H_{g,1})=r$ is the largest positive root of the equation:
$$\frac{x+\sqrt{x^{2}-4}}{2}(x^2-1)-2x=0, \; {\mbox{i.e.},}\;x^{4}-4x^{2}-1=0.$$
Thus, $\lim\limits_{g \to \infty}\mu(H_{g,1})=\tau^{\frac{1}{2}}+\tau^{-\frac{1}{2}}.$	
\end{proof}

\begin{lem}\label{3.3}
	Let $H_{g,t}$ be the lollipop graph with the labelling of Fig. 4. Then for every $g\geq 3$,
	$$\lim_{t \to \infty}\mu(H_{g,t})\geq\tau^{\frac{1}{2}}+\tau^{-\frac{1}{2}}.$$
\end{lem}
\begin{proof}
	Since $H_{g,t} \subset H_{g,t+1}$, by Lemma \ref{largest root}(ii) we get $\mu(H_{g,t}) <\mu(H_{g,t+1})<2\sqrt{2}$. Easy to check $\mu (H_{3,2})=2$. Hence, $\lim\limits_{t \to \infty}\mu(H_{g,t})=r_{g} > 2$ exists by Lemma \ref{largest root}(ii) again. On the other hand, from Lemma \ref{3.1}(i) it follows that $\mu(H_{g+1,t})> \mu(H_{g,t})$ which indicates
	\begin{equation}\label{6}
		\lim_{t \to \infty}\mu(H_{g,t}) \geq \lim_{t \to \infty}\mu(H_{3,t}).
	\end{equation}
	We now consider the sequence $\{\mu(H_{3,t}) | t \in \mathbb{N}\}$. Let $v\in V(C_{3})$ with $d_{H_{3,t}}(v)=3$. In light of Lemma \ref{k largest root equation} (ii), $r_{3}$ is the largest root of the equation:
	\begin{equation}
		\frac{x+\sqrt{x^{2}-4}}{2}\mathscr{M}(C_{3},x)-\mathscr{M}(C_{3}-v,x)=0.\label{7}
	\end{equation}
	Substituting $\mathscr{M}(C_{3},x) = x^{3}-3x$ and $\mathscr{M}(C_{3}-v,x) = x^{2}-1$ into \eqref{7}, then	
	$$\frac{x+\sqrt{x^{2}-4}}{2}(x^{3}-3x)-(x^{2}-1)=0, \; {\mbox{i.e.},}\;x^{4}-4x^{2}-1=0.$$
	Thereby, $\lim\limits_{t \to \infty}\mu(H_{3,t})=\tau^{\frac{1}{2}}+\tau^{-\frac{1}{2}}$, and our result holds from \eqref{6}.
\end{proof}

To obtain this claim, we discuss  the girth $g_{i}$ of $G_{i}$. Let $\mathcal{S}=\{g_{i}\,\vert\, i\geq1\}$. If $g_{i}$ is unbounded, then each $G_{i}$ contains a cycle $C_{g_i} \subseteq G_i$ with $g_i> g_j$ for $i> j$. Thus, $H_{g_i,1} \subseteq G_i$. Along with Lemma \ref{3.2}, we get $\lim\limits_{i \to \infty}\mu(G_i)\geq \lim\limits_{i \to \infty}\mu(H_{g_i,1})=\tau^{\frac{1}{2}}+\tau^{-\frac{1}{2}},$ a contradiction. If $g_{i}$ is bounded and there are infinitely many $G_{i}$’s which are not trees, then each $G_{i}$ contains a cycle $C_{g}$ for some fixed $g\geq 3$. Consequently, $H_{g,k_{i}} \subseteq G_i$ with $k_{i}> k_{j}$ for $i>j$. In the light of Lemma \ref{3.3}, it turns out that $\lim\limits_{i \to \infty}\mu(G_i)\geq \lim\limits_{i \to \infty}\mu(H_{g,k_{i}})\geq\tau^{\frac{1}{2}}+\tau^{-\frac{1}{2}},$ a contradiction again.

As proved above, to determine the limit points of the largest matching roots of graphs smaller than $\tau^{\frac{1}{2}}+\tau^{-\frac{1}{2}}$, it is enough to analyse the tree sequence $\{G_i\, |\, i\in \mathbb{N}\}$ with maximum degree $\Delta(G_i) =3$. In the next section, we will describe the concrete structure of such trees.

\section{Restrictions on trees}

Let $T_{i,j,k}$ be a tree with exactly one vertex $u$ of maximum degree three such that $T_{i,j,k}-u=P_{i}\cup P_{j}\cup P_{k}$, which is shown in Fig. 5. In this section, we prove that the tree sequence $\{G_i\, |\, i\in \mathbb{N}\}$ mentioned in Section 2 will be $\{T_{1,n,m_{i}} |\, i\in \mathbb{N}\}$.

\begin{center}
	\setlength{\unitlength}{1mm} \linethickness{1pt}
	\begin{picture}(80,50)
		\multiput(10,20)(8,0){7}{\circle*{1.5}}
		\multiput(34,20)(0,8){4}{\circle*{1.5}} \linethickness{1pt}
		
		\linethickness{0.6pt}
		\put(10,20){\line(3,0){10}}\put(24,20){\line(3,0){20}}
		\put(48,20){\line(3,0){10}}\put(34,20){\line(0,3){10}}
		\put(34,34){\line(0,3){10}}
		\multiput(18,20)(1.5,0){6}{\line(1,0){0.5}}
		\multiput(42,20)(1.5,0){6}{\line(1,0){0.5}}
		\multiput(34,28)(0,1.5){6}{\line(0,1){0.5}}
		\put(8,16){${i}$} \put(14,16){${i-1}$}
		\put(25,16){${1}$} \put(33,16){${u}$}
		
		\put(41,16){${1}$} \put(46,16){${j-1}$}
		\put(58,16){${j}$}
		\put(36,27){${1}$}
		\put(36,35){${k-1}$} \put(36,43){${k}$}
		\put(8,6){{\bf Fig. 5:} The graph $T_{i,j,k}(i,j,k\geq1)$.}
	\end{picture}
\end{center}

\vspace{-6mm}

To achieve our goal, we prove the following preparatory lemmas.

\begin{lem}\label{exchange limpoint}
	Let $G_{1}$ and $G_{2}$ be two vertex-disjoint connected graphs with $u \in V(G_1)$ and $v \in V(G_2)$. $W_{uv}(G_{1},G_{2};P_{n})$ is the graph obtained from $G_{1}$, $G_{2}$ and a new path $P_{n}$: $v_{1}v_{2}\cdots v_{n-1}v_{n}$ by adding two new edges $uv_{1}$ and $vv_{n}$, respectively. Then
	\begin{equation*}
		\lim_{n \to \infty}\mu(W_{uv}(G_{1},G_{2};P_{n}))=\max \left\{\lim_{n \to \infty}\mu(L_{u}(G_{1};P_{n})),\lim_{n \to \infty}\mu(L_{v}(G_{2};P_{n}))\right\}.
	\end{equation*}
\end{lem}

\begin{proof}
	In view of Lemma \ref{k largest root equation}(i), both $\lim\limits_{n \to \infty}\mu(L_{u}(G_{1};P_{n}))$ and $\lim\limits_{n \to \infty}\mu(L_{v}(G_{2};P_{n}))$ exist. The following cases are to be considered:
	
	{\it Case 1.} $uv_{1}v_{2}\cdots v_{n}v$ is not in an internal path.
	
	Either $L_{u}(G_{1};P_{n})$ or $L_{v}(G_{2};P_{n})$ is a path. Thereby $W_{uv}(G_{1},G_{2};P_{n})\subset W_{uv}(G_{1},G_{2};P_{n+1})$, and so $\mu(W_{uv}(G_{1},G_{2};P_{n}))<\mu(W_{uv}(G_{1},G_{2};P_{n+1}))$ by Lemma \ref{largest root}(ii). Hence, we have
	\begin{equation*}
		\lim_{n \to \infty}\mu(W_{uv}(G_{1},G_{2};P_{n}))=\max \left\{\lim_{n \to \infty}\mu(L_{u}(G_{1};P_{n})),\lim_{n \to \infty}\mu(L_{v}(G_{2};P_{n}))\right\}.
	\end{equation*}
	
	{\it Case 2.} $uv_{1}v_{2}\cdots v_{n}v$ is in an internal path.
	
	Since $uv_{1}$ is a cut edge in $W_{uv}(G_{1},G_{2};P_{n})$, then $\mu(W_{uv}(G_{1},G_{2};P_{n}))\geq\mu(W_{uv}(G_{1},G_{2};P_{n+1}))$ from Lemma \ref{3.1}(ii). Therefore, we get $\lim\limits_{n \to \infty}\mu(W_{uv}(G_{1},G_{2};P_{n}))$ exists.
	
	Notice,  $P_{n}\subset L_{u}(G_{1};P_{n})\subset W_{uv}(G_{1},G_{2};P_{n})$ and $P_{n}\subset L_{v}(G_{2};P_{n})\subset W_{uv}(G_{1},G_{2};P_{n})$. If $\lim\limits_{n \to \infty}\mu(W_{uv}(G_{1},G_{2};P_{n}))=2$, by Lemma \ref{largest root}(ii)  we have
	\begin{equation*}
		\lim\limits_{n \to \infty}\mu(L_{u}(G_{1};P_{n}))=\lim\limits_{n \to \infty}\mu(L_{v}(G_{2};P_{n}))=\lim\limits_{n \to \infty}\mu(W_{uv}(G_{1},G_{2};P_{n}))=2,
	\end{equation*}
	and thus
	\begin{equation*}
		\lim_{n \to \infty}\mu(W_{uv}(G_{1},G_{2};P_{n}))=\max \left\{\lim_{n \to \infty}\mu(L_{u}(G_{1};P_{n})),\lim_{n \to \infty}\mu(L_{v}(G_{2};P_{n}))\right\}.
	\end{equation*}
	If $\lim\limits_{n \to \infty}\mu(W_{uv}(G_{1},G_{2};P_{n}))>2$, set $G = W_{uv}(G_{1},G_{2};P_{n})$. In the light of Equality \eqref{0},
	\begin{align*}
		\mathscr{M}(G)=&x\mathscr{M}(G-u)-\sum_{w\in N_{G}(u)}\mathscr{M}(G-u-w)\\[2mm]
		=&x\mathscr{M}(G_{1}-u)\mathscr{M}(L_{v}(G_{2};P_{n}))-\mathscr{M}(G_{1}-u)\mathscr{M}(L_{v}(G_{2};P_{n})-v_{1})\\[2mm]
		&-\sum_{w\in N_{G_{1}}(u)}\mathscr{M}(G_{1}-u-w)\mathscr{M}(L_{v}(G_{2};P_{n}))\\[2mm] =&\mathscr{M}(G_{1})\mathscr{M}(L_{v}(G_{2};P_{n}))-\mathscr{M}(G_{1}-u)\mathscr{M}(L_{v}(G_{2};P_{n-1}))\\[2mm]
		=&\mathscr{M}(G_{1})\left[\mathscr{M}(G_{2})\mathscr{M}(P_{n})-\mathscr{M}(G_{2}-v)\mathscr{M}(P_{n-1})\right]\\[2mm] &-\mathscr{M}(G_{1}-u)\left[\mathscr{M}(G_{2})\mathscr{M}(P_{n-1})-\mathscr{M}(G_{2}-v)\mathscr{M}(P_{n-2})\right]\\[2mm]
		=&\left[\mathscr{M}(G_{1})\left(\mathscr{M}(G_{2})\dfrac{\mathscr{M}(P_{n})}{\mathscr{M}(P_{n-2})}-\mathscr{M}(G_{2}-v)\dfrac{\mathscr{M}(P_{n-1})}{\mathscr{M}(P_{n-2})}\right)\right.\\[2mm]
		&\left.-\mathscr{M}(G_{1}-u)\left(\mathscr{M}(G_{2})\dfrac{\mathscr{M}(P_{n-1})}{\mathscr{M}(P_{n-2})}-\mathscr{M}(G_{2}-v)\right)\right]\mathscr{M}(P_{n-2})\\[2mm]	
		\triangleq& \mathscr{M}(P_{n-2})f_{n}(x).
	\end{align*}
	For $x>2$, from Lemma \ref{2.9} we have
	\begin{align*}
		\lim_{n \to \infty}f_{n}(x)=&\mathscr{M}(G_{1})\left(\mathscr{M}(G_{2})\lim_{n \to \infty}\dfrac{\mathscr{M}(P_{n})}{\mathscr{M}(P_{n-2})}-\mathscr{M}(G_{2}-v)\lim_{n \to \infty}\dfrac{\mathscr{M}(P_{n-1})}{\mathscr{M}(P_{n-2})}\right)\\[2mm]
		&-\mathscr{M}(G_{1}-u)\left(\mathscr{M}(G_{2})\lim_{n \to \infty}\dfrac{\mathscr{M}(P_{n-1})}{\mathscr{M}(P_{n-2})}-\mathscr{M}(G_{2}-v)\right)\\[2mm]
		=&\mathscr{M}(G_{1})\left(a^{2}\mathscr{M}(G_{2})-a\mathscr{M}(G_{2}-v)\right)-\mathscr{M}(G_{1}-u)\left(a\mathscr{M}(G_{2})-\mathscr{M}(G_{2}-v)\right)\\[2mm]
		=&\left(\frac{x+\sqrt{x^{2}-4}}{2}\mathscr{M}(G_{1})-\mathscr{M}(G_{1}-u)\right)
		\left(\frac{x+\sqrt{x^{2}-4}}{2}\mathscr{M}(G_{1})-\mathscr{M}(G_{2}-v)\right)\\[2mm]
		\triangleq& f(x).			
	\end{align*}
	Since $\lim\limits_{n \to \infty}\mu(G)>2$, we conclude that $\lim\limits_{n \to \infty}\mu(G)$ is the largest positive root of the equation $f(x)=0$. Along with Lemma \ref{k largest root equation}, we have
	\begin{equation*}
		\lim_{n \to \infty}\mu(W_{uv}(G_{1},G_{2};P_{n}))=\max \left\{\lim_{n \to \infty}\mu(L_{u}(G_{1};P_{n})),\lim_{n \to \infty}\mu(L_{v}(G_{2};P_{n}))\right\}.
	\end{equation*}
	This finishes our proof.
\end{proof}

\begin{lem}\label{T1ii}
Let $T_{i,j,k}$ be a tree defined in Fig. 5. Then
\begin{itemize}
\item[{\rm (i)}]
For the tree $T_{1,i,i}$, $\lim\limits_{i \to \infty}\mu(T_{1,i,i})=\tau^{\frac{1}{2}}+\tau^{-\frac{1}{2}}$.
\item[{\rm (ii)}]
For the tree $T_{2,2,i}$, $\lim\limits_{i \to \infty}\mu(T_{2,2,i})=\tau^{\frac{1}{2}}+\tau^{-\frac{1}{2}}$.	
\end{itemize}	
\end{lem}	
	
\begin{proof}
(i) Due to $T_{1,i,i} \subset T_{1,i+1,i+1}$ and $\mu(T_{1,3,3})=2$, it follows from Lemma \ref{largest root}(ii) that $2<\mu(T_{1,i,i})<\mu(T_{1,i+1,i+1})<2\sqrt{2}$ for $i>3$. Thus  $\lim\limits_{i \to \infty}\mu(T_{1,i,i})=r > 2$ exists. Due to Lemma \ref{k largest root equation}(ii), $r$ is the largest root of the equation:
$$\frac{x+\sqrt{x^{2}-4}}{4}\mathscr{M}(P_{2},x)-\mathscr{M}(P_{1},x)=0.$$
Since $\mathscr{M}(P_{2},x)=x^{2}-1$ and $\mathscr{M}(P_{1},x)=x$, then above equation is equivalent to
$$\frac{x+\sqrt{x^{2}-4}}{4}(x^{2}-1)-x=0, \; {\mbox{i.e.},}\;x^{4}-4x^{2}-1=0,$$
which implies $\lim\limits_{i \to \infty}\mu(T_{1,i,i})=\tau^{\frac{1}{2}}+\tau^{-\frac{1}{2}}$.

(ii) Since $T_{2,2,i} \subset T_{2,2,i+1}$ and $\mu(T_{2,2,2})=2$, by Lemma \ref{largest root}(ii) we have $2<\mu(T_{2,2,i})<\mu(T_{2,2,i+1})<2\sqrt{2}$ for $i>2$, and thus $\lim\limits_{i \to \infty}\mu(T_{2,2,i})=r > 2$. In view of (ii), $r$ is the largest root of the equation:
$$\frac{x+\sqrt{x^{2}-4}}{2}\mathscr{M}(P_{5},x)-\mathscr{M}(2P_{2},x)=0.$$
Easily to verify that $\mathscr{M}(P_{5},x)=x^5-4x^3+3x$ and $\mathscr{M}(2P_{2},x)=(x^2-1)^2$. Thereby, the previous equation is equivalent to

$$\frac{x+\sqrt{x^{2}-4}}{2}(x^5-4x^3+3x)-(x^2-1)^2=0, \; {\mbox{i.e.},}\;x^{4}-4x^{2}-1=0,$$
whose largest root is $\tau^{\frac{1}{2}}+\tau^{-\frac{1}{2}}$. Thus, (ii) follows.		
\end{proof}

Recall, the maximum degree of each tree in the sequence $\{G_i\, |\, i\in \mathbb{N}\}$ is 3. We next inspect the number of vertices with degree 3 in such a tree.  If $G_{i}$ has at least three vertices of degree three, then $R_{p_{i},q_{i}}\subseteq G_{i}$ (see Fig. 6). When $k_{i}$ is sufficiently large, it follows from Lemmas \ref{largest root}(ii) and \ref{3.1}(ii) that $\mu(G_{i})\geq \mu(R_{p_{i},q_{i}})\geq \mu(T_{1,k_{i},k_{i}})$. Along with Lemma \ref{T1ii}(i), we have $\lim\limits_{i \to \infty}\mu(G_{i})\geq \lim\limits_{i \to \infty}\mu(T_{1,k_{i},k_{i}})=\tau^{\frac{1}{2}}+\tau^{-\frac{1}{2}}$, a contradiction.

\begin{center}
	\setlength{\unitlength}{1mm} \linethickness{1pt}
	\begin{picture}(110,40)
		\multiput(55,25)(0,8){2}{\circle*{1.5}} \linethickness{1pt}
		\multiput(23,25)(8,0){9}{\circle*{1.5}}
		
		\multiput(15,20)(0,10){2}{\circle*{1.5}} \linethickness{1pt}
		\multiput(95,20)(0,10){2}{\circle*{1.5}} \linethickness{1pt}
		\linethickness{0.6pt}	
		
		\put(55,25){\line(0,1){8}}
		\put(37,25){\line(3,0){36}} \put(23,25){\line(3,0){10}}
		\put(77,25){\line(3,0){10}}
		\put(23,25){\line(-5,3){8}} \put(23,25){\line(-5,-3){8}}
		\put(87,25){\line(5,-3){8}} \put(87,25){\line(5,3){8}}
		\multiput(31,25)(1.5,0){8}{\line(1,0){0.5}}
		\multiput(71,25)(1.5,0){8}{\line(1,0){0.5}}
		
		\put(62,21){$u_{_1}$} \put(70,21){$u_{_2}$}
		\put(78,21){$u_{_{p-1}}$}\put(86,21){$u_{_p}$}
		
		\put(46,21){$v_{_1}$} \put(38,21){$v_{_2}$}
		\put(30,21){$v_{_{q-1}}$}\put(23,21){$v_{_q}$}
		\put(29,8){{\bf Fig. 6:} The graph $R_{p,q}$ $(p,q\geq1)$.}
	\end{picture}
\end{center}
	
 \vspace{-10mm}
	
If $G_{i}$ has two vertices of degree three and their distance is bounded, then $F_{s,r_{i}}\subseteq G_{i}$ (see Fig. 7) with $r_{i}>r_{j}$ for $i>j$, and thus $\mu(G_{i})\geq \mu(F_{s,r_{i}})\geq \mu(T_{1,r_{i},r_{i}})$ by Lemmas \ref{largest root}(ii) and \ref{3.1}(ii). Along with Lemma \ref{T1ii}(i), we have
$\lim\limits_{i \to \infty}\mu(G_{i})\geq \lim\limits_{i \to \infty}\mu(T_{1,r_{i},r_{i}})=\tau^{\frac{1}{2}}+\tau^{-\frac{1}{2}}$, a contradiction.

\begin{center}
	\setlength{\unitlength}{1mm} \linethickness{1pt}
	\begin{picture}(110,40)
		\multiput(15,20)(0,10){2}{\circle*{1.5}}
		\multiput(23,25)(9,0){2}{\circle*{1.5}}
		\multiput(63,25)(9,0){2}{\circle*{1.5}}\multiput(83,25)(9,0){2}{\circle*{1.5}}
		
		\multiput(45,25)(9,0){2}{\circle*{1.5}}
		\multiput(54,25)(0,8){2}{\circle*{1.5}} \linethickness{1pt}

		\put(15,20){\line(3,2){8}} \put(23,25){\line(-3,2){8}}
		\linethickness{0.6pt} \put(23,25){\line(1,0){11}}
		\multiput(33,25)(1.5,0){9}{\line(1,0){0.5}}
		\put(43,25){\line(1,0){31}}

		\put(22,21){$v_{_1}$} \put(31,21){$v_{_2}$}
		\put(43,21){$v_{_{s-1}}$}
		\put(53,21){$v_{_{s}}$}
		\multiput(75,25)(1.5,0){8}{\line(1,0){0.5}}
		\put(81,25){\line(1,0){11}}
		\put(54,25){\line(0,1){8}}
		\put(62,21){$u_{_1}$} \put(71,21){$u_{_2}$}
		\put(82,21){$u_{_{r-1}}$} \put(91,21){$u_{_{r}}$}
		
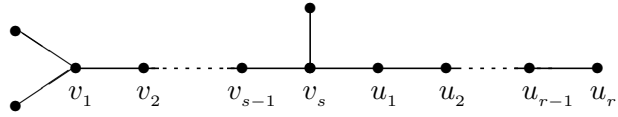
		\put(25,8){{\bf Fig. 7:} The graph $F_{s,r}$ $(s>1,r\geq1)$.}
	\end{picture}
\end{center}

\vspace{-10mm}

If $G_{i}$ has two vertices of degree three and their distance is unbounded, by Lemma \ref{exchange limpoint} we can obtain a tree sequence $\{T_{k,n,m_{i}}\}_{i=1}^{\infty}$ with $m_{i}>m_{j}$ for $i>j$ such that
$\lim\limits_{i \to \infty}\mu(T_{k,n,m_{i}})=\lim\limits_{i \to \infty}\mu(G_{i})$. Without loss of generality, set $m_i \geq n \geq k \geq 1$. If $k \geq 2$,
then $T_{2,2,m_{i}} \subseteq T_{k,n,m_i}$. Along with Lemma \ref{T1ii}(ii) we get $\lim\limits_{i \to \infty}\mu(G_{i})=\lim\limits_{i \to \infty}\mu(T_{k,n,m_{i}})\geq\lim\limits_{i
\to \infty}\mu(T_{2,2,m_{i}})=\tau^{\frac{1}{2}}+\tau^{-\frac{1}{2}},$ a contradiction. Thus, $k=1$, and so the desired tree sequence is
$\{T_{1,n,m_i} |\, i \in \mathbb{N}\}$ (see Fig. 8).
	
\begin{center}
	 \setlength{\unitlength}{1mm} \linethickness{1pt}
	\begin{picture}(80,35)
		\multiput(10,20)(8,0){7}{\circle*{1.5}}
		\multiput(34,20)(0,8){2}{\circle*{1.5}} \linethickness{1pt}
		\linethickness{0.6pt}
		
		\put(10,20){\line(3,0){10}}\put(24,20){\line(3,0){20}}
		\put(48,20){\line(3,0){10}}\put(34,20){\line(0,3){8}}	
		
		\multiput(18,20)(1.5,0){6}{\line(1,0){0.5}}
		\multiput(42,20)(1.5,0){6}{\line(1,0){0.5}}

		\put(36,28){$_{1}$}
		\put(9,17){$_{n}$} \put(16,17){$_{n-1}$}
		\put(25,17){${_{1}}$} \put(33,17){$_{u}$}
		
		\put(41,17){${_1}$} \put(48,17){${_{m_{i}-1}}$}
		\put(57,17){${_{m_{i}}}$}

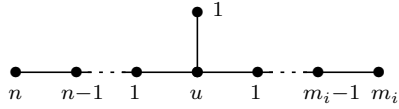
		\put(15,6){{\bf Fig. 8:} The tree $T_{1,n,m_{i}}$.}
	\end{picture}
\end{center}

\vspace{-10mm}

\section{Completion of proof}	

As discussed in Section 3, the limit points that we seek are the numbers $\lim\limits_{i \to \infty}\mu(T_{1,n,m_{i}})=\alpha_{n}$. Due to Lemma \ref{largest root}(ii) and Lemma \ref{k largest root equation}, $\alpha_{n}$ is the largest root of the equation:	
\begin{equation*}
	\frac{x+\sqrt{x^{2}-4}}{2}\mathscr{M}(P_{n+2},x)-\mathscr{M}(P_{1},x)\mathscr{M}(P_{n},x)=0.
\end{equation*}
Let $x=\lambda +\frac{1}{\lambda}$ and $\lambda^{2}=z$. In view of Lemma \ref{2.9}, we get
\begin{align*}
	&\frac{x+\sqrt{x^{2}-4}}{2}\mathscr{M}(P_{n+2},x)-\mathscr{M}(P_{1},x)\mathscr{M}(P_{n},x)\\
	&=\frac{x+\sqrt{x^{2}-4}}{2\sqrt{x^{2}-4}}(a^{n+3}-b^{n+3})-\frac{x}{\sqrt{x^{2}-4x}}(a^{n+1}-b^{n+1})\\ &=\frac{\lambda^{2}}{\lambda^{2}-1}(\lambda^{n+3}-\dfrac{1}{\lambda^{n+3}})-\dfrac{\lambda^{2}+1}{\lambda^{2}-1}(\lambda^{n+1}-\dfrac{1}{\lambda^{n+1}})\\
	&=\frac{1}{\lambda^{n-1}}\left(\dfrac{\lambda^{2(n+2)}-\lambda^{2(n+1)}-\lambda^{2n}+1}{\lambda^{2}-1}\right)\\
	&=\frac{1}{\lambda^{n-1}}\left(\lambda^{2(n+1)}-\dfrac{\lambda^{2n}-1}{\lambda^{2}-1}\right)\\
	&=\frac{1}{\lambda^{n-1}}\left[\lambda^{2(n+1)}-(1+\lambda^{2}+\cdots+\lambda^{2(n-1)})\right]\\
	&=\frac{1}{\lambda^{n-1}}\left[z^{n+1}-(1+z+\cdots+z^{n-1})\right].
\end{align*}

Set $h_{n}(x):=x^{n+1}-(1+x+\cdots+x^{n-1})$. Since $h_{n}(x)-xh_{n-1}(x)=-1<0$ for $x\geq 0$, then $\beta_{n+1}>\beta_{n}$ where $\beta_{n}$ is the largest root of $h_{n}(x)$. Hence, the sequence $\{\alpha_{n}\}_{n=1}^{\infty}$ with $\alpha_{n}=\beta_{n}^{\frac{1}{2}}+\beta_{n}^{-\frac{1}{2}}$ is strictly increasing. It follows from Lemma \ref{T1ii}(i)  that $\lim\limits_{n \to \infty}\alpha_{n}=\tau^{\frac{1}{2}}+\tau^{-\frac{1}{2}}$. Together with Theorem \ref{ii}, this finishes our proof of Theorem \ref{main}. \hfill{$\Box$} \medskip

In this paper, we determine all limit points of the largest matching roots of graphs. Note that in Sections 2 we have shown that the graph sequence $\{G_i\, |\, i\in \mathbb{N}\}$ is composed of trees. To obtain the limit points of adjacency spectral radii of graphs less than $\tau^{\frac{1}{2}}+\tau^{-\frac{1}{2}}$, Hoffman \cite{Ho A J} utilized the same tree sequence as the one considered in this paper. Due to \eqref{mu=rho}, we can get the limit points of the largest matching roots of graphs less than $\tau^{\frac{1}{2}}+\tau^{-\frac{1}{2}}$ from Hoffman's result. For the completeness of this paper, we provide a detailed proof in Sections 3 and 4; moreover, the newly discovered properties enrich the theory of matching polynomials.

\section*{Data Availability}

There is no data associated with this article.

\section*{Declaration of Interest Statement}

The authors declare that they have no known competing financial interests or personal relationships that could have appeared to influence the work reported in this paper.

\section*{Acknowledgement}
Zhaoxi Li and Jianfeng Wang are supported by National Natural Science Foundation of China (No. 12371353).

\small{

}

\begin{thebibliography}{00}
	
	\bibitem{Aih}
	J. Aihara, A new definition of Dewar-type resonance energy, J. Amer. Chem. Soc. 98 (1976) 2750--2758.
	
	\bibitem{bal}
	K. Balasubramanian,  Matching Polynomial-Based Similarity Matrices and Descriptors for Isomers of Fullerenes, Inorganics 11 (2023) 335.
	
	\bibitem{Brouwer}
	A. E. Brouwer, W. H. Haemers, Spectra of Graphs, Springer Science $\&$ Business Media, 2011.
	
	\bibitem{bur-erey}
	M. Burnham, A. Erey,  On the roots of maximal matching polynomials, Discrete Math. 348 (2025) 114547.
	
	\bibitem{chen-yuan}
	H. Chen, Y. Yuan, Matching polynomials of path-trees of a complete bipartite graph, Discrete Appl. Math. 372 (2025) 173--179.
	
	
	\bibitem{book}
	D. Cvetkovi\'c, P. Rowlinson, S. Simi\'c, An Introduction to the Theory of Graph Spectra, London
	Math. Soc. Stud. Texts, vol. 75, Cambridge University Press, Cambridge, 2010.
	
	
	
	
	\bibitem{far}
	E.J. Farrell, An introduction to matching polynomials, J. Combin. Theory Ser. B 27 (1979) 75--86.
	
	
	\bibitem{C.D.Godsil2}
	
	C.D. Godsil, Matchings and walks in graphs, J. Graph Theory 5 (1981) 285--297.
	
	\bibitem{god-book}
	C.D. Godsil, Algebraic Combinatorics, Chapman and Hall, London, 1993.
	
	
	
	\bibitem{god-gut}
	C.D. Godsil, I. Gutman, On the matching polynomial of a graph, in: Algebraic Methods in Graph Theory, Vol. I, II (Szeged, 1978), in: Colloq. Math. Soc. J\'{a}nos Bolyai, 25, North-Holland, Amsterdam, New York, 1981, pp. 241--249.
	
	\bibitem{god-gut-JGT}
	C.D. Godsil, I. Gutman, On the theory of the matching polynomial, J. Graph Theory 5 (1981) 137--144.
	
	\bibitem{gru-kun}
	C. Gruber, H. Kunz, General properties of polymer systems, Commun. Math. Phys. 22 (1971) 133--161.
	
	
	
	\bibitem{chen-yuan1}
	M. Guo, H. Chen, The matching polynomial of the path-tree of a complete graph, Discrete Appl. Math. 359 (2024) 244--249.
	
	
	
	
	\bibitem{Gutman1}
	I. Gutman, Uniqueness of the matching polynomial, MATCH Commun. Math. Comput. Chem. 55 (2006) 351--358.
	
	
	
	\bibitem{gut-mil-tri1}
	I. Gutman, M. Milun, and N. Trinajsti\'c. Topological definition of delocalisation
	energy. MATCH Commun. Math. Comput. Chem. 1 (1975) 171--175.
	
	
	\bibitem{gut-mil-tri}
	I. Gutman, M. Milun, M.N. Trinajsti\'c, Non-parametric resonance energies of arbitrary conjugated systems, J. Amer. Chem. Soc. 99 (1977) 1692--1704.
	
	
	
	
	\bibitem{hei-lie}
	O.J. Heilmann, E.H. Lieb, Theory of monomer-dimer systems, Comm. Math. Phys. 25 (1972) 190--232.
	
	\bibitem{Ho A J}
	A.J. Hoffman, On limit points of spectral radii of non-negative symmetric integral matrices, Graph Theory and Appl. Lect. Notes Math. 303 (1972) 165--172.
	
	\bibitem{Ho and Sm}
	A.J. Hoffman, J.H. Smith, On the spectral radii of topologically equivalent graphs, in: M. Fiedler (Ed.), Recent Advances in Graph Theory, Academia Praha, 1975, pp. 273--281.
	
	
	\bibitem{jiang-poly}
	Z.L. Jiang, A. Polyanskii, Forbidden subgraphs of bounded spectral radius with applications to equiangular lines, Israel  J.  Math. 236 (2020) 393--421.
	
	\bibitem{jiang-tid-yao-zhang}
	Z.L. Jiang, J. Tidor, Y. Yao, S.T. Zhang, Y.F. Zhao, Equiangular lines with a fixed angle, Ann. Math. 194 (2021) 729--743.
	
	\bibitem{kun}
	H. Kunz, Location of the zeros of the partition function for some classical lattice
	systems, Phys. Lett. 32 (1970) 311--312.
	
	\bibitem{ku-chen}
	C.Y. Ku, W. Chen, An analogue of the Gallai–Edmonds Structure Theorem for non-zero
	roots of the matching polynomial, J. Combin. Theory Ser. B 100 (2010) 119--127.
	
	
	
	\bibitem{mar-spi-sri}
	A.W. Marcus, D.A. Spielman, N. Srivastava, Interlacing families I: Bipartite Ramanujan graphs of all degrees. Ann. Math. 182 (2015) 307--325.
	
	
	\bibitem{she}
	J.B. Shearer, On the distribution of the maximum eigenvalue of graphs, Linear Algebra Appl. 114/115 (1989) 17--20.
	
	\bibitem{shi-GP}
	Y.T. Shi, M. Dehmer, X.L. Li, I. Gutman, Graph Polynomials, Chapman $\&$ Hall New York, 2017.
	
	\bibitem{yan-yeh}
	W. Yan, Y.-N. Yeh, On the matching polynomial of subdivision graphs,  Discrete Appl. Math. 157 (2009) 195--200.
	
	\bibitem{3w-2b}
	J.F. Wang, J. Wang, M. Maurizio, F. Belardo, L.G. Wang, Developments on the Hoffman program of graph, Adv. Appl. Math. 169 (2025) 10291.
	
	
\end{thebibliography}
\end{document}